\documentclass{amsart}

\usepackage{amsmath}
\usepackage{amssymb}
\usepackage{amsthm}
\usepackage{esint}
\usepackage{enumitem}
\usepackage{url}

\theoremstyle{plain}
\newtheorem{theorem}{Theorem}[section]
\newtheorem{corollary}[theorem]{Corollary}
\newtheorem{lemma}[theorem]{Lemma}
\newtheorem{proposition}[theorem]{Proposition}

\theoremstyle{definition}
\newtheorem{definition}[theorem]{Definition}
\newtheorem{remark}[theorem]{Remark}

\theoremstyle{remark}

\numberwithin{theorem}{section}
\numberwithin{equation}{section}

\newcommand{\N}{\mathbb{N}}

\newcommand{\R}{\mathbb{R}}

\newcommand{\dist}{\mathrm{dist}}
\newcommand{\diam}{\mathrm{diam}}

\newcommand{\cl}{\overline}
\newcommand{\del}{\partial}
\newcommand{\loc}{\mathrm{loc}}

\newcommand{\laplacian}{\Delta}

\DeclareMathOperator*{\essinf}{ess\,inf}
\DeclareMathOperator*{\spt}{supp}
\newcommand{\capacity}{\mathrm{cap}}

\newcommand{\trinorm}[1]
{{
    \left\vert\kern-0.20ex\left\vert\kern-0.20ex\left\vert
    #1 
    \right\vert\kern-0.20ex\right\vert\kern-0.20ex\right\vert
}}

\newcommand{\M}{\mathcal{M}}
\newcommand{\e}{\mathcal{E}}
\newcommand{\W}{{\bf{W}}}
\newcommand{\w}{\mathcal{W}}
\newcommand{\s}{{S}}

\begin{document}





\title[Quasilinear elliptic equations with sub-natural growth terms]
	{Quasilinear elliptic equations with sub-natural growth terms in bounded domains}
\author{Takanobu Hara}
\address{%
	Department of Mathematics \\
	Hokkaido University \\
	Kita 8 Nishi 10  Sapporo \\
	Hokkaido 060-0810, Japan
}
\email{takanobu.hara.math@gmail.com}
\date{\today}




\begin{abstract}
We consider the existence of positive solutions to
weighted quasilinear elliptic differential equations of the type
\[
\begin{cases}
- \Delta_{p, w} u = \sigma u^{q}
&
\text{in $\Omega$},
\\
u = 0
&
\text{on $\partial \Omega$}
\end{cases}
\]
in the sub-natural growth case $0 < q < p - 1$, where $\Omega$ is a bounded domain in $\mathbb{R}^{n}$,
$\Delta_{p, w}$ is a weighted $p$-Laplacian, and $\sigma$ is a nonnegative (locally finite) Radon measure on $\Omega$.
We give criteria for the existence problem.
For the proof, we investigate various properties of $p$-superharmonic functions, especially the solvability of Dirichlet problems with infinite measure data.
\end{abstract}




\maketitle


\section{Introduction}\label{sec:introduction}

Let $\Omega$ be a bounded domain in $\R^{n}$ and let $1 < p < \infty$.
We consider the existence of positive solutions to quasilinear elliptic equations of the type
\begin{equation}\label{eqn:p-laplace}
\begin{cases}
\displaystyle
- \laplacian_{p, w} u = \sigma u^{q}
&
\text{in $\Omega$},
\\
u = 0
&
\text{on $\del \Omega$,}
\end{cases}
\end{equation}
in the sub-natural growth case $0 < q < p - 1$,
where
$\laplacian_{p, w}$ is a weighted $(p, w)$-Laplacian,
$w$ is a $p$-admissible weight on $\R^{n}$ (see Section \ref{sec:preliminaries} below)
and $\sigma$ is a nonnegative (locally finite) Radon measure on $\Omega$.

For the standard theory of sublinear equations,
we refer to \cite{MR0181881, MR679140, MR820658, MR1141779} and the references therein.
In the classical existence results of weak solutions, the boundedness of coefficients was assumed.
Boccardo and Orsina \cite{MR1272564} removed this assumption
and pointed out that if the integrability of the coefficients is low,
the solutions do not necessarily belong to the class of weak solutions.
Therefore, we interpret this equation in the sense of $p$-superharmonic functions, or locally renormalized solutions.
For details on such generalized solutions, especially on the relation between the two concepts, see \cite{MR1955596, MR2859927, MR3676369} and references therein.
Hereafter, we use the framework for $p$-superharmonic functions.

The measure-valued coefficient equation \eqref{eqn:p-laplace} is relevant to the following $L^{p}$-$L^{1 + q}$ trace inequality:
\begin{equation}\label{eqn:trace_ineq}
\| f \|_{L^{1 + q}(\Omega; \sigma)} \le C_{T} \| \nabla f \|_{L^{p}(\Omega; w)}, \quad \forall f \in C_{c}^{\infty}(\Omega).
\end{equation}
Maz'ya and Netrusov \cite{MR1313906} gave a capacitary condition that characterizes \eqref{eqn:trace_ineq}.
Cascante, Ortega and Verbitsky \cite{MR1734322, MR1747901} studied non-capacitary characterizations for inequalities of the type \eqref{eqn:trace_ineq}.
For example, if $\Omega = \R^{n}$ and $w = 1$, then the best constant $C_{T}$ in \eqref{eqn:trace_ineq} satisfies
\[
\frac{1}{c} \, C_{T}^{ \frac{(1 + q)p}{p - 1 - q} }
\le
\int_{\R^{n}}
\left( \W_{1, p} \sigma \right)^{\frac{(1 + q)(p - 1)}{p - 1 - q}}
\, d \sigma
\le
c \, C_{T}^{ \frac{(1 + q)p}{p - 1 - q} },
\]
where $c = c(n, p, q)$ and $\W_{1, p} \sigma$ is the \textit{Wolff potential} of $\sigma$ (see \cite{MR0409858, MR727526}).

Recently, Verbitsky and his colleagues studied the problem of
existence of solutions to elliptic equations related to \eqref{eqn:p-laplace}
and presented some criteria
(see  \cite{MR3567503, MR3311903, MR3556326, MR3938014, MR4105916, MR3642745, MR3724493, MR3985926, MR3881877, MR3792109, VERBITSKY2019111516}).
In their study, they treated the cases of $\Omega = \R^{n}$, or $p = 2$.

For $\Omega = \R^{n}$, Wolff potentials are suitable potentials for the problem.
Every $p$-superharmonic function known to be locally estimated by Wolff potentials (see \cite{MR1205885,MR1264000}),
and these estimates are some of the key pieces of proof.
In contrast, we can directly use Green potentials for $p = 2$ or, more generally, for linear equations.
However, Wolff potentials are not sufficient for estimating the boundary behavior of solutions.
Furthermore, Radon measures satisfying \eqref{eqn:trace_ineq} may not have compact support in $\Omega$ and are not even finite in general (see Section \ref{sec:example2}).
Consequently, the complete criteria for the existence of solutions to quasilinear equations in bounded domains have not yet been obtained.

The purpose of this paper is to extend Verbitsky's theory to (weighted) quasilinear equations in bounded domains.
Our basic idea is to replace Wolff or Green potentials with
the minimal positive $p$-superharmonic solution $u$ to $- \laplacian_{p, w} u = \sigma$ (see Definition \ref{def:potential} for the precise meaning).
The function $\w_{p, w} \sigma = u$ has no explicit integral representation; however, some required estimates can be obtained by directly using the properties of weak solutions.
To realize this approach, we investigate various properties of $p$-superharmonic functions, especially the solvability of Dirichlet problems in the case of infinite measure data.
Note that this existence problem has been stated as an open problem in \cite[Problem 2]{MR1990293}.

Let $\M^{+}_{0}(\Omega)$ be the set of all nonnegative Radon measures on $\Omega$
that are absolutely continuous with respect to the $(p, w)$-capacity.
Note that $\sigma$ must belong to $\M^{+}_{0}(\Omega)$ if \eqref{eqn:trace_ineq} holds.
Our main result is as follows.

\begin{theorem}\label{thm:main_theorem}
Let $\Omega$ be a bounded domain in $\R^{n}$.
Let $1 < p < \infty$ and $0 < q < p - 1$.
Suppose that $\sigma \in \M^{+}_{0}(\Omega) \setminus \{ 0 \}$.
Fix $0 < \gamma < \infty$.
Then the following statements are equivalent:
\begin{enumerate}[label=(\arabic*)]
\item\label{enum:01@main_theorem}
There exists a nontrivial nonnegative $(p, w)$-superharmonic supersolution $v$ to $- \laplacian_{p, w} v = \sigma v^{q}$ in $\Omega$
satisfying $\| v \|_{L^{\gamma + q}(\sigma)} \le C_{1} < \infty$.
\item\label{enum:02@main_theorem}
The measure $\sigma$ satisfies
\begin{equation}\label{eqn:energy_cond}
\left(
\int_{\Omega} \left( \w_{p, w} \sigma \right)^{ \frac{(\gamma + q)(p - 1)}{p - 1 - q} } \, d \sigma
\right)^{\frac{1}{\gamma + q}}
\le C_{2} < \infty.
\end{equation}
\item\label{enum:03@main_theorem}
The following weighted norm inequality holds:
\begin{equation}\label{eqn:weighted_norm_ineq}
\| \w_{p, w}( |f| \sigma ) \|_{L^{\gamma + q}( \sigma )}
\le
C_{3} \| f \|_{L^{ \frac{\gamma + q}{q} }( \sigma )}^{\frac{1}{p - 1}},
\quad
\forall f \in L^{ \frac{\gamma + q}{q} }( \sigma ).
\end{equation}
\end{enumerate}
Moreover, if $C_{i}$ ($i = 1, 2 ,3$) are the best constants in the above statements, then
\[
C_{1}
\le
C_{3}^{\frac{p - 1}{p - 1 - q}}
\le
c_{E}^{\frac{1}{p - 1 - q}} \, C_{2}
\le
\frac{ c_{E}^{\frac{1}{p - 1 - q}} }{ c_{V} } \, C_{1},
\]
where
\begin{equation}\label{eqn:opt-const}
c_{E}
:=
\left( \frac{p - 1 + \gamma}{p} \right)^{p} \frac{1}{\gamma},
\quad
c_{V}
:=
\left( \frac{p - 1 - q}{p - 1} \right)^{ \frac{p - 1}{p - 1 - q} }.
\end{equation}
In addition, if one of the above statements holds,
then there exists a minimal positive $(p, w)$-superharmonic solution $u$ to $- \laplacian_{p, w} u = \sigma u^{q}$ in $\Omega$
such that $\| u \|_{L^{\gamma + q}(\sigma)} \le C_{1}$.
\end{theorem}

The boundary behavior of $u$ is not discussed in Theorem \ref{thm:main_theorem}; however, under appropriate assumptions on $\sigma$,
we can prove that $u$ vanishes on $\del \Omega$ (see Proposition \ref{prop:energy_class} and Corollary \ref{cor:energy_esti}).
One such condition is as follows.

\begin{theorem}\label{thm:main_theorem_FE}
Let $\Omega$ be a bounded domain in $\R^{n}$.
Let $1 < p < \infty$ and $0 < q < p - 1$.
Suppose that $\sigma \in \M^{+}_{0}(\Omega) \setminus \{ 0 \}$.
Then there exists a unique positive weak solution $u \in H_{0}^{1, p}(\Omega; w)$ to \eqref{eqn:p-laplace}
if and only if
\begin{equation}\label{eqn:energy_cond@FE}
\int_{\Omega} \left( \w_{p, w} \sigma \right)^{ \frac{(1 + q)(p - 1)}{p - 1 - q} } \, d \sigma
<
\infty.
\end{equation}
Moreover,
\[
c_{V}^{1 + q}
\int_{\Omega} \left( \w_{p, w} \sigma \right)^{ \frac{(1 + q)(p - 1)}{p - 1 - q} } \, d \sigma
\le
\| \nabla u \|_{L^{p}(w)}^{p}
\le
\int_{\Omega} \left( \w_{p, w} \sigma \right)^{ \frac{(1 + q)(p - 1)}{p - 1 - q} } \, d \sigma.
\]
\end{theorem}

We also give the Cascante-Ortega-Verbitsky type of theorem below.

\begin{theorem}\label{thm:embedding}
Let $\Omega$ be a bounded domain in $\R^{n}$.
Let $1 < p < \infty$ and let $-1 < q < p - 1$.
Suppose that $C_{T}$ is the best constant of \eqref{eqn:trace_ineq}.
Then
\[
C_{T}^{ \frac{(1 + q)p}{p - 1 - q} }
\le
\int_{\Omega} \left( \w_{p, w} \sigma \right)^{ \frac{(1 + q)(p - 1)}{p - 1 - q} } \, d \sigma
\le
\left( \frac{1}{1 + q} \right)^{ \frac{1 + q}{p - 1 - q} } \frac{1}{c_{V}^{1 + q}} \,
C_{T}^{ \frac{(1 + q)p}{p - 1 - q} }.
\]
In particular, if \eqref{eqn:trace_ineq} holds with some $-1 < q < p - 1$,
then the equation $- \laplacian_{p, w} u = \sigma$ in $\Omega$ has 
a minimal positive $(p, w)$-superharmonic solution.
\end{theorem}

\begin{remark}
The constants in Theorems \ref{thm:main_theorem}-\ref{thm:embedding}
do not depend on $n$ or the data of $w$.
Due to the qualitative arguments in the proof, $w$ must be $p$-admissible; however, its quantitative properties, especially the Sobolev-type inequalities are not used.
\end{remark}

In particular, if \eqref{eqn:trace_ineq} holds with $q > 0$, then there exists a unique positive weak solution $u$ to \eqref{eqn:p-laplace} such that
\[
\frac{1}{c(p, q)} C_{T}^{ \frac{1 + q}{p - 1 - q} } \le\| \nabla u \|_{L^{p}(w)} \le c(p, q) C_{T}^{ \frac{1 + q}{p - 1 - q} }.
\]

Examples of concrete sufficient conditions for \eqref{eqn:energy_cond}
will be discussed at the end of the paper.
One is a Lorentz scale refinement of \cite[Theorem 5.5]{MR1272564},
and the other is quasilinear ordinary differential equations with nonintegrable Hardy-type coefficients.

\subsection*{Organization of the paper}
Section \ref{sec:preliminaries} presents various facts in nonlinear potential theory and introduces classes of smooth measures.
Section \ref{sec:potentials} discusses minimal $p$-superharmonic solutions to Dirichlet problems.
Section \ref{sec:generalized_energy} defines the generalized energy of $p$-superharmonic functions
and investigates its properties including Theorem \ref{thm:embedding}.
Section \ref{sec:proof_of_main_theorem} provides the proof of Theorem \ref{thm:main_theorem}
by using results in Sections \ref{sec:potentials} and \ref{sec:generalized_energy}.
Sections \ref{sec:example1} and \ref{sec:example2} discuss two applications of Theorems \ref{thm:main_theorem}-\ref{thm:embedding}.

\subsection*{Acknowledgments}
The author would like to thank Professor Verbitsky for suggesting references to added to an earlier version of the manuscript.
This work was supported by JSPS KAKENHI Grant Number JP18J00965.

\subsection*{Notation}
We use the following notation.
Let $\Omega$ be a domain (connected open subset) in $\R^{n}$. 
\begin{itemize}
\item
$\mathbf{1}_{E}(x) :=$ the indicator function of a set $E$.
\item
$C_{c}^{\infty}(\Omega) :=$
the set of all infinitely differentiable functions with compact support in $\Omega$.
\item
$\M^{+}(\Omega) :=$ the set of all nonnegative Radon measures on $\Omega$.
\item
$L^{p}(\mu) :=$ the $L^{p}$ space with respect to $\mu \in \M^{+}(\Omega)$.
\end{itemize}
For a ball $B = B(x, R)$ and $\lambda > 0$, $\lambda B := B(x, \lambda R)$.
For measures $\mu$ and $\nu$, we denote $\nu \le \mu$ if $\mu - \nu$ is a nonnegative measure.
For a sequence of extended real valued functions $\{ f_{j} \}_{j = 1}^{\infty}$,
we denote $f_{j} \uparrow f$
if $f_{j + 1} \ge f_{j}$ for all $j \ge 1$ and $\lim_{j \to \infty} f_{j} = f$.
Moreover, $c$ and $C$ denote various constants with and without indices.

\section{Preliminaries}\label{sec:preliminaries}

We first recall the basic properties of $p$-admissible weights from \cite[Chapter A.2]{MR2867756}, \cite[Chapter 20]{MR2305115} and references therein.
Throughout the paper, $1 < p < \infty$ is a fixed constant.
A Lebesgue measurable function $w$ on $\R^{n}$
to be said as \textit{weight} on $\R^{n}$ if $w \in L^{1}_{\loc}(\R^{n}; dx)$ and $w(x) > 0$ $dx$-a.e.
We write $w(E) = \int_{E} w \, dx$ for a Lebesgue measurable set $E \subset \R^{n}$.
We always assume that $w$ is \textit{$p$-admissible},
that is, positive constants $C_{D}$, $C_{P}$ and $\lambda \ge 1$ exist, such that
\[
w(2B) \le C_{D} w(B)
\]
and
\[
\fint_{B} |f - f_{B}| \, dw \le C_{P} \, \diam(B) \left( \fint_{\lambda B} |\nabla f|^{p} \, dw \right)^{\frac{1}{p}},
\quad
\forall f \in C_{c}^{\infty}(\R^{n}),
\]
where $B$ is an arbitrary ball in $\R^{n}$,
$\fint_{B} = w(B)^{-1} \int_{B}$
and
$f_{B} = \fint_{B} f \, dw$.
One of the important properties of $p$-admissible weights is the Sobolev inequality (\cite{MR1098839, MR1150597}).
In particular, the following form of the Poincar\'{e} inequality holds:
\begin{equation*}\label{eqn:poincare}
\int_{B} |f|^{p} \, dw
\le
C \, \diam(B)^{p} \int_{B} |\nabla f|^{p} \, dw,
\quad
\forall f \in C_{c}^{\infty}(B),
\end{equation*}
where $C$ is a constant depending only on $p$, $C_{D}$, $C_{P}$ and $\lambda$.

Next, we recall basics of nonlinear potential theory from \cite[Chapters 1-10 and 21]{MR2305115}.
Let $\Omega$ be a bounded domain in $\R^{n}$.
The weighted Sobolev space $H^{1, p}(\Omega; w)$ is the closure of $C^{\infty}(\Omega)$
with respect to the norm
\[
\| u \|_{H^{1, p}(\Omega; w)}
:=
\left(
\int_{\Omega} |u|^{p} + |\nabla u|^{p} \, d w
\right)^{\frac{1}{p}},
\]
where $\nabla u$ is the gradient of $u$.
The corresponding local space $H^{1, p}_{\loc}(\Omega; w)$ is defined in the usual manner.
We denote the closure of $C_{c}^{\infty}(\Omega)$ in $H^{1, p}(\Omega; w)$ by $H_{0}^{1, p}(\Omega; w)$.
Since $\Omega$ is bounded, we can take $\| \nabla \cdot \|_{L^{p}(\Omega; w)}$ 
as the norm of $H_{0}^{1, p}(\Omega; w)$ by the Poincar\'{e} inequality.

For $u \in H^{1, p}_{\loc}(\Omega; w)$,
we define the weighted $p$-Laplace operator $\laplacian_{p, w}$ by
\[
\langle - \laplacian_{p, w} u, \varphi \rangle
=
\int_{\Omega} |\nabla u|^{p - 2} \nabla u \cdot \nabla \varphi \, dw,
\quad
\forall \varphi \in C_{c}^{\infty}(\Omega).
\]
A function $u \in H^{1, p}_{\loc}(\Omega; w)$ is called a \textit{supersolution} to
\begin{equation}\label{eqn:p-harmonic}
- \laplacian_{p, w} u = 0 \quad \text{in} \ \Omega
\end{equation}
if $\langle - \laplacian_{p, w} u, \varphi \rangle \geq 0$ for all nonnegative $\varphi \in C_{c}^{\infty}(\Omega)$.
If $\mu$ is an element of the dual of $H_{0}^{1, p}(\Omega; w)$,
then the Dirichlet problem
\begin{equation}\label{eqn:variational_problem}
\begin{cases}
\langle - \laplacian_{p, w} u, \varphi \rangle = \langle \mu, \varphi \rangle
\quad
\forall \varphi \in H_{0}^{1, p}(\Omega; w)
\\
u \in H_{0}^{1, p}(\Omega; w)
\end{cases}
\end{equation}
has a unique weak solution $u$.

Let $\Omega \subset \R^{n}$ be open, and let $K \subset \Omega$ be compact.
The \textit{variational $(p, w)$-capacity} $\capacity_{p, w}(K, \Omega)$
of the condenser $(K, \Omega)$ is defined by
\[
\capacity_{p, w}(K, \Omega)
:=
\inf \left\{
\| \nabla u \|_{L^{p}(\Omega; w)}^{p} \colon u \geq 1 \ \text{on} \ K, \ u \in C_{c}^{\infty}(\Omega)
\right\}.
\]
Moreover, for $E \subset \Omega$, we define
\[
\capacity_{p, w}(E, \Omega)
:=
\inf_{\substack{ E \subset U \subset \Omega \\ U \colon \text{open} }}
\sup_{K \subset U \colon \text{compact}}
\capacity_{p , w}(K, \Omega).
\]
Since $\Omega \subset \R^{n}$ is bounded,
$\capacity_{p, w}(E, \Omega) = 0$ if and only if $C_{p, w}(E) = 0$,
where $C_{p, w}(\cdot)$ is the \textit{(Sobolev) capacity} of $E$.
We say that a property holds \textit{quasieverywhere} (q.e.)
if it holds except on a set of $(p, w)$-capacity zero.
An extended real valued function $u$ on $\Omega$ is called as \textit{quasicontinuous}
if for every $\epsilon > 0$ there exists an open set $G$ such that
$C_{p, w}(G) < \epsilon$ and $u|_{\Omega \setminus G}$ is continuous.
Every $u \in H^{1, p}_{\loc}(\Omega; w)$ has a quasicontinuous representative $\tilde{u}$
such that $u = \tilde{u}$ a.e.

A function $u \colon \Omega \to ( - \infty, \infty]$ is called \textit{$(p, w)$-superharmonic} if
$u$ is lower semicontinuous in $\Omega$, is not identically infinite,
and satisfies the comparison principle on each subdomain $D \Subset \Omega$;
if $h \in C(\cl{D})$ is a continuous weak solution to $- \laplacian_{p, w} u = 0$ in $D$,
and if $u \geq h$ on $\del D$, then $u \geq h$ in $D$.
If $u$ is a bounded $(p, w)$-superharmonic function,
then $u$ belongs to $H^{1, p}_{\loc}(\Omega; w)$ and is a supersolution to \eqref{eqn:p-harmonic}.
Conversely, if $u$ is a supersolution to \eqref{eqn:p-harmonic},
then its \textit{lsc-regularization}
\[
u^{*}(x)
=
\lim_{r \to 0} \essinf_{B(x, r)} u
\]
is $(p, w)$-superharmonic in $\Omega$.
If $u$ and $v$ are $(p, w)$-superharmonic in $\Omega$ and $u(x) \leq v(x)$ for a.e. $x \in \Omega$,
then $u(x) \leq v(x)$ for all $x \in \Omega$.
Every $(p, w)$-superharmonic function is known to be quasicontinuous.
In particular, the set $\{ u = \infty \}$ has zero $(p, w)$-capacity whenever $u$ is $(p, w)$-superharmonic.

Assume that $u$ is a $(p, w)$-superharmonic function in $\Omega$.
Its truncation $u_{k} = \min\{ u, k \}$ will then become a supersolution to \eqref{eqn:p-harmonic} for all $k > 0$.
Moreover, there exists a Radon measure $\mu[u]$ such that
\[
\lim_{k \to \infty}
\int_{\Omega} |\nabla u_{k}|^{p - 2} \nabla u_{k} \cdot \nabla \varphi \, dw
=
\int_{\Omega} \varphi \, d \mu[u],
\quad
\forall \varphi \in C_{c}^{\infty}(\Omega).
\]
The measure $\mu[u]$ is called the \textit{Riesz measure} of $u$.
By definition, if $u \in H^{1, p}_{\loc}(\Omega; w)$, then $\mu[u] = - \laplacian_{p, w} u$ in the sense of distribution.

As in Section \ref{sec:introduction},
we denote by $\M^{+}_{0}(\Omega)$ the set of all Radon measures $\mu$
that are absolutely continuous with respect to the $(p, w)$-capacity, i.e.,
$\mu(E) = 0$ whenever $E$ has zero $(p, w)$-capacity.
If $u \in H^{1, p}_{\loc}(\Omega; w)$ is $(p, w)$-superharmonic in $\Omega$,
then the Riesz measure of $u$ belongs to $\M^{+}_{0}(\Omega)$.
It is known that if $\mu \in \M^{+}_{0}(\Omega)$ is finite,
then the integral $\int_{\Omega} f \, d \mu$ is well-defined for any quasicontinuous function $f$ on $\Omega$.

If $u \in H^{1, p}(\Omega; w)$ is a supersolution to \eqref{eqn:p-harmonic},
then the Riesz measure of $u^{*}$ satisfies
\begin{equation}\label{eqn:finite_energy}
\int_{\Omega} \varphi \, d \mu
\le
C \| \varphi \|_{H^{1, p}(\Omega; w)},
\quad \forall \varphi \in C_{c}^{\infty}(\Omega), \ \varphi \ge 0.
\end{equation}
If $\mu$ is also finite, then we can replace $C_{c}^{\infty}(\Omega)$ with $H_{0}^{1, p}(\Omega; w)$ up to taking a quasicontinuous representative,
i.e., the dual action of $\mu$ has the integral representative.
Conversely, if a finite measure $\mu$ satisfies \eqref{eqn:finite_energy},
then there exists a unique weak solution $u \in H_{0}^{1, p}(\Omega; w)$ to
\[
\int_{\Omega} |\nabla u|^{p - 2} \nabla u \cdot \nabla \varphi \, dw
=
\int_{\Omega} \tilde{ \varphi } \, d \mu
\quad
\forall \varphi \in H_{0}^{1, p}(\Omega; w).
\]

The following weak continuity result was given by Trudinger and Wang \cite{MR1890997}.

\begin{theorem}[{\cite[Theorem 3.1]{MR1890997}}]\label{thm:TW}
Suppose that $\{ u_{k} \}_{k = 1}^{\infty}$ is a sequence of nonnegative $(p, w)$-superharmonic functions in $\Omega$.
Assume that $u_{k} \to u$ a.e. in $\Omega$ and that $u$ is $(p, w)$-superharmonic in $\Omega$.
Let $\mu[u_{k}]$ and $\mu[u]$ be the Riesz measures of $u_{k}$ and $u$, respectively.
Then $\mu[u_{k}]$ converges to $\mu[u]$ weakly, that is
\[
\int_{\Omega} \varphi \, d \mu[u_{k}] \to \int_{\Omega} \varphi \, d \mu[u],
\quad
\forall \varphi \in C_{c}^{\infty}(\Omega).
\]
\end{theorem}
The following Harnack-type convergence theorem follows from combining Theorem \ref{thm:TW} and \cite[Lemma 7.3]{MR2305115}:
If $\{ u_{k} \}_{k = 1}^{\infty}$ is a nondecreasing sequence of $(p, w)$-superharmonic functions in $\Omega$
and if $u := \lim_{k \to \infty} u_{k} \not \equiv \infty$, 
then $u$ is $(p, w)$-superharmonic in $\Omega$ and $\mu[u_{k}]$ converges to $\mu[u]$ weakly.

The following form of Wolff potential estimate was first established by
Kilpel\"{a}inen and Mal\'{y} \cite{MR1205885,MR1264000}.
The extension to weighted equations is due to Mikkonenn \cite{MR1386213}.
See also \cite{MR1890997,MR3842214} for other proofs.

\begin{theorem}[{\cite[Theorem 3.1]{MR1386213}}]\label{thm:wolff}
Suppose that $u$ is a nonnegative $(p, w)$-superharmonic function in $B(x, 2R)$.
Let $\mu$ be the Riesz measure of $u$.
Then
\[
\frac{1}{C} \W_{1, p, w}^{R} \mu(x)
\leq
u(x)
\leq
C \left(
\inf_{B(x, R)} u
+
\W_{1, p, w}^{2R} \mu(x)
\right),
\] 
where $C \ge 1$ is a constant depending only on $p$, $C_{D}$, $C_{P}$ and $\lambda$, and
$\W_{1, p, w}^{R} \mu$ is the \textit{truncated Wolff potential} of $\mu$, which is defined by
\[
\W_{1, p, w}^{R} \mu (x)
:=
\int_{0}^{R}
\left(
r^{p} \frac{ \mu(B(x, r))}{ w(B(x, r)) }
\right)^{\frac{1}{p - 1}}
\frac{dr}{r}.
\]
\end{theorem}

\section{Minimal $p$-superharmonic solutions}\label{sec:potentials}

The following comparison principle improves \cite[Lemma 5.2]{MR3567503}.
We do not assume the finiteness of the Riesz measures here.

\begin{theorem}\label{thm:comparison_principle}
Let $\Omega$ be a bounded domain in $\R^{n}$.
Let $u$ and $v$ be nonnegative $(p, w)$-superharmonic functions in $\Omega$
with the Riesz measures $\mu$ and $\nu$, respectively.
Assume also that $\mu \le \nu$ and $u \in H_{0}^{1, p}(\Omega; w)$.
Then $u(x) \leq v(x)$ for all $x \in \Omega$.
\end{theorem}

\begin{proof}
For each $k \in \N$, set $v_{k} = \min\{ v, k \}$. Let $\nu_{k}$ be the Riesz measure of $v_{k}$.
By the chain rule of Sobolev functions,
\[
\Psi^{k}_{\epsilon}(v) 
:=
\frac{ (k - v)_{+} }{ (k - v)_{+} + \epsilon }
\in H^{1, p}_{\loc}(\Omega; w) \cap L^{\infty}(\Omega),
\quad
\forall \epsilon > 0.
\]
Thus, for any $l \ge k$,
\[
\begin{split}
\int_{\Omega} \varphi \, \Psi^{k}_{\epsilon}(v) \, d \nu_{l}
& =
\int_{\Omega} |\nabla v_{l}|^{p - 2} \nabla v_{l} \cdot \nabla ( \varphi \, \Psi^{k}_{\epsilon}(v) ) \, dw
\\
& \le
\int_{\Omega} |\nabla v_{l}|^{p - 2} \nabla v_{l} \cdot \nabla \varphi \, \Psi^{k}_{\epsilon}(v)  \, dw,
\quad
\forall \varphi \in C_{c}^{\infty}(\Omega), \ \varphi \ge 0.
\end{split}
\]
Passing to the limits $\epsilon \to 0$ and $l \to \infty$, we see that $\mathbf{1}_{ \{ v < k \} } \nu \le \nu_{k}$.
Let $\mu_{k} = \mathbf{1}_{ \{ v < k \} } \mu$,
and let $\{ \Omega_{k} \}_{k = 1}^{\infty}$ be a sequence of open sets 
such that $\bigcup_{k = 1}^{\infty} \Omega_{k} = \Omega$ and $\Omega_{k} \Subset \Omega_{k + 1}$ for all $k \ge 1$.
Since $\mu_{k} \le \mu$,
there exists a weak solution $u_{k} \in H_{0}^{1, p}(\Omega_{k}; w)$ to
$- \laplacian_{p, w} u_{k} = \mu_{k}$ in $\Omega_{k}$.
Let us denote by $u_{k}$ the zero extension of the lsc-regularization of $u_{k}$ again. 
By the comparison principle for weak solutions,
$u_{k}(x) \le v_{k}(x) \le v(x)$ for a.e. $x \in \Omega_{k}$.
Accordingly, $u_{k}(x) \le v(x)$ for all $x \in \Omega_{k}$ since $u$ and $v$ are $(p, w)$-superharmonic in $\Omega_{k}$.
Similarly, $u_{k} \le u_{k + 1} \le u$ in $\Omega_{k}$.
Let $u' = \lim_{k \to \infty} u_{k}$, and let $\mu'$ be the Riesz measure of $u'$.
Testing the equation of $u_{k}$ with $u_{k}$, we have
\[
\int_{\Omega} |\nabla u_{k}|^{p} \, d w
=
\int_{\Omega_{k}} u_{k} \, d \mu_{k}
\le
\int_{\Omega_{k}} u_{k} \, d \mu
\le
\| \nabla u \|_{L^{p}(\Omega; w)}^{p - 1}
\| \nabla u_{k} \|_{L^{p}(\Omega; w)}.
\]
Therefore $u' \in H_{0}^{1, p}(\Omega; w)$.
By Theorem \ref{thm:TW}, $\{ \mu_{k} \}_{k = 1}^{\infty}$ converges to $\mu'$ weakly.
Meanwhile, since $\mu( \{ v = \infty \} ) = 0$, $\{ \mu_{k} \}_{k = 1}^{\infty}$  converges to $\mu$ weakly.
Thus $\mu = \mu'$.
By the uniqueness of weak solutions, $u = u' \le v$ in $\Omega$.
\end{proof}

Following \cite[Chapter 2]{MR2778606},
we introduce classes of smooth measures. 

\begin{definition}
Let $\s_{0}$ be the set of all Radon measures satisfying \eqref{eqn:finite_energy}.
For $\mu \in \s_{0}$,
we denote the lsc-regularization of  the solution to \eqref{eqn:variational_problem} by $\w_{p, w}^{0} \mu$.
Furthermore, we define a subset $\s_{00}$ of $\s_{0}$ as
\[
\s_{00}
:=
\left\{
\mu \in \s_{0} \colon \sup_{\Omega} \w_{p, w}^{0} \mu < \infty \ \text{and} \ \mu(\Omega) < \infty
\right\}.
\]
\end{definition}

\begin{remark}\label{rem:wmp}
Assume that $\mu \in S_{0}$ is finite. Set $u = \w_{p, w}^{0} \mu$.
Then $\sup_{\Omega} u = \| u \|_{L^{\infty}(\mu)}$.
In fact, clearly $\sup_{\Omega} u \ge \| u \|_{L^{\infty}(\mu)}$.
To prove the converse inequality, set $k = \| u \|_{L^{\infty}(\mu)}$.
Then testing the equation of $u$ with $(u - k)_{+}$, we find that
\[
\int_{ \{ u > k \} } |\nabla u|^{p} \, dw
=
\int_{\Omega} (u - k)_{+} \, d \mu
=
0.
\]
Therefore $u(x) \le k$ for a.e. $x \in \Omega$.
Since $u$ is $(p, w)$-superharmonic in $\Omega$,
$u(x) \le k$ for all $x \in \Omega$. 
\end{remark}

\begin{remark}
If $\mu \in \s_{00}$ and $f \in L^{\infty}(\mu)$, then $|f| \mu \in \s_{00}$.
\end{remark}

\begin{lemma}\label{lem:additivity_s_00}
Let $\mu, \nu \in \s_{00}$.
Assume also that $\spt(\mu + \nu) \Subset \Omega$.
Then $\mu + \nu \in \s_{00}$.
\end{lemma}

\begin{proof}
Clearly $\mu + \nu$ is finite and belongs to $\s_{0}$.
Let $u = \w_{p, w}^{0}(\mu + \nu)$.
By Remark \ref{rem:wmp}, $\sup_{\Omega} u = \| u \|_{L^{\infty}(\mu + \nu)}$.
Let $R = \dist(\spt (\mu + \nu), \del \Omega) / 4$ and fix $x \in \spt (\mu + \nu)$.
Then, by the latter inequality in Theorem \ref{thm:wolff}, we have
\[
u(x)
\le
C \left( \inf_{B(x, R)} u + \W_{1, p, w}^{2R} (\mu + \nu)(x)  \right).
\]
Using a simple calculation and the former inequality in Theorem \ref{thm:wolff}, we also obtain
\[
\begin{split}
\W_{1, p, w}^{2R} (\mu + \nu)(x)
& \le
\max\{ 2^{ \frac{2 - p}{p - 1}}, 1 \} \left( \W_{1, p, w}^{2R}\mu(x) + \W_{1, p, w}^{2R}\nu(x) \right)
\\
& \le
C \left( \w_{p, w}^{0} \mu(x) + \w_{p, w}^{0} \nu(x) \right).
\end{split}
\]
Furthermore, testing the equation of $u$ with $\min \{ u, \inf_{B(x, R)} u \}$
and using the Poincar\'{e} inequality, we find that
\[
\inf_{B(x, R)} u
\le
\left(
\frac{ (\mu + \nu)(\Omega) }{ \capacity_{p, w}( \cl{ B(x, R) }, \Omega) }
\right)^{\frac{1}{p - 1}}
\le
C
\left(
\diam(\Omega)^{p} \frac{ (\mu + \nu)(\Omega) }{ w(B(x, R)) }
\right)^{\frac{1}{p - 1}}.
\]
The right-hand side is continuous with respect to $x$.
Hence it is bounded on $\spt( \mu + \nu )$.
Combining these estimates, we obtain the desired boundedness.
\end{proof}

The following characterization for $\M^{+}_{0}(\Omega)$ is
the nonlinear counterpart of \cite[Theorem 2.2.4]{MR2778606} (see also \cite[Corollary 3.19]{MR0386032}).
For related characterizations for $\M_{0}^{+}(\Omega)$,
see also \cite{MR740203} and \cite[Proposition 1.2.7]{MR3676369}.

\begin{theorem}\label{thm:approximation}
Let $\mu \in \M^{+}(\Omega)$.
Then, $\mu \in \M^{+}_{0}(\Omega)$ if and only if
there exists an increasing sequence of compact sets $\{ F_{k } \}_{k = 1}^{\infty}$ such that
$\mu_{k} := \mathbf{1}_{F_{k}} \mu \in \s_{00}$ for all $k \ge 1$
and
$\mu\left( \Omega \setminus \bigcup_{k = 1}^{\infty} F_{k} \right) = 0$.
\end{theorem}

\begin{proof}
The ``if'' part is readily obtained from the definition of $\s_{0}$. Let us prove the ``only if'' part.
Consider a sequence of open sets $\{ \Omega_{j} \}_{j = 1}^{\infty}$ such that 
$\bigcup_{j = 1}^{\infty} \Omega_{j} = \Omega$ and $\Omega_{j} \Subset \Omega_{j + 1}$ for all $j \ge 1$.
For each $j \ge 1$, $\mathbf{1}_{ \cl{\Omega_{j}} } \mu$ is finite.
Thus, by \cite[Theorem 6.6]{MR1386213},
there exists a $(p, w)$-superharmonic function $u_{j}$ satisfying
\[
\begin{cases}
- \laplacian_{p, w} u_{j} = \mathbf{1}_{ \cl{\Omega_{j}} } \mu & \text{in $\Omega$,}
\\
\min\{ u_{j}, k \} \in H_{0}^{1, p}(\Omega; w) & \text{for all $k \ge 1$.}
\end{cases}
\]
Let $F_{k, j} = \{ u_{j} \le k \}$ and let $\mu_{j, k} = \mathbf{1}_{ F_{k, j} \cap \cl{\Omega_{j}} } \mu$.
Using the method in Theorem \ref{thm:comparison_principle},
we see that $\mu_{j, k} \le \mathbf{1}_{ \{ u _{j} < k + 1\} } \mathbf{1}_{ \cl{\Omega_{j}} } \mu \le \mu[ \min\{ u_{j}, k + 1 \} ]$.
Since $\min \{ u_{j}, k + 1 \} \in H_{0}^{1, p}(\Omega; w)$, this implies that $\mu_{j, k} \in \s_{0}$.
By Remark \ref{rem:wmp},
\[
\sup_{\Omega} \w_{p, w}^{0} \mu_{j, k}
=
\| \w_{p, w}^{0} \mu_{j, k} \|_{L^{\infty}( \mu_{j, k} )}
\le
\| \min \{ u_{j}, k + 1 \} \|_{L^{\infty}( \mu_{j, k} )}
\le
k.
\]
Hence  $\mu_{j, k} \in \s_{00}$.
Let
\[
F_{k} = \bigcup_{1 \le j \le k} \left( F_{k, j} \cap \cl{ \Omega_{j} } \right).
\]
Clearly, $F_{k} \subset F_{k + 1}$ and $F_{k} \subset \cl{ \Omega_{k} }$.
Fix $j \ge 1$. Since $\mu \in \M^{+}_{0}(\Omega)$, we have $\mu(\{ u_{j} = \infty \}) = 0$.
Therefore,
\[
\mathbf{1}_{ \cl{ \Omega_{j} } }
=
\lim_{k \to \infty} \mathbf{1}_{ F_{k, j} \cap \cl{ \Omega_{j} } }
\le
\lim_{k \to \infty} \mathbf{1}_{ F_{k} }
\le
\mathbf{1}_{\Omega}
\quad
\text{ $\mu$-a.e. }
\]
Passing to the limit $j \to \infty$, we see that $\lim_{k \to \infty} \mathbf{1}_{ F_{k} } = \mathbf{1}_{\Omega}$ $\mu$-a.e.
It remains to be shown that $\mathbf{1}_{F_{k}} \mu \in \s_{00}$.
By the comparison principle for weak solutions,
\[
\sup_{\Omega} \w_{p, w}^{0} \left( \mathbf{1}_{F_{k}} \mu \right)
\le
\sup_{\Omega} \w_{p, w}^{0} \left( \sum_{1\le j \le k} \mu_{j, k} \right).
\]
Iterating Lemma \ref{lem:additivity_s_00} $(k - 1)$ times, we find that the right-hand side is finite.
\end{proof}

Let us now consider the following Dirichlet problem:
\begin{equation}\label{eqn:poisson}
\begin{cases}
- \laplacian_{p, w} u = \mu & \text{in $\Omega$},
\\
u = 0 & \text{on $\del \Omega$.}
\end{cases}
\end{equation}
We say that a function $u$ is a \textit{$(p, w)$-superharmonic solution (supersolution)}
to $- \laplacian_{p, w} u = \mu$ in $\Omega$,
if $u$ is a $(p, w)$-superharmonic function in $\Omega$ and $\mu[u] = \mu$ ($\mu[u] \ge \mu$),
where $\mu[u]$ is the Riesz measure of $u$.
We say that a nontrivial nonnegative solution $u$ is \textit{minimal}
if $v \geq u$ in $\Omega$ whenever
$v$ is a nontrivial nonnegative supersolution to the same equation.

\begin{definition}\label{def:potential}
For $\mu \in \M^{+}_{0}(\Omega)$,
we define
\[
\w_{p, w} \mu(x)
:=
\sup \left\{ \w_{p, w}^{0} \nu(x) \colon \nu \in \s_{00} \ \text{and} \  \nu \le \mu \right\}.
\]
\end{definition}

\begin{remark}
Note that $u := \w_{p, w} \mu$ may be identically infinite.
If $\mu$ is finite, then $u$ is finite q.e. in $\Omega$ and satisfies \eqref{eqn:poisson}
in the sense of entropy or renormalized solutions
(see \cite{MR1205885, MR1409661,MR1402674, MR1386213, MR1760541}).
In general, even if $u \not \equiv \infty$ is a $p$-superharmonic solution to $- \laplacian_{p, w} u = \mu$ in $\Omega$,
it does not satisfy the Dirichlet boundary condition in the sense of renormalized solutions.
The author is not aware of the renowned name for this class of solutions.
Sufficient conditions for $u \not \equiv \infty$ will be discussed in the next section.
\end{remark}

\begin{remark}
Clearly, $\w_{p, w}( a \mu) = a^{\frac{1}{p - 1}} \w_{p, w} \mu$ for any constant $a \ge 0$, and
\[
\mu \le \nu
\Rightarrow
\w_{p, w} \mu(x) \le \w_{p, w} \nu(x),
\quad
\forall x \in \Omega.
\]
Furthermore, Theorem \ref{thm:comparison_principle} is still valid when $u = \w_{p, w} \mu$.
\end{remark}

\begin{proposition}\label{prop:potentials}
Let $\mu \in \M^{+}_{0}(\Omega)$.
The following statements hold.
\begin{enumerate}[label=(\roman*)]
\item\label{cond:01_potentials}
Assume that $\w_{p, w} \mu \not \equiv \infty$.
Then $u = \w_{p, w} \mu$ is the minimal nonnegative $(p, w)$-superharmonic solution to
$- \laplacian_{p, w} u = \mu$ in $\Omega$.
\item\label{cond:02_potentials}
Let $\{ f_{j} \}_{j = 1}^{\infty} \subset L^{1}_{\loc}(\mu)$ be a nondecreasing sequence of functions.
Assume that $f_{j} \uparrow f$ $\mu$-a.e.
Let $u_{j} = \w_{p, w} ( f_{j} \mu)$ and let $u = \lim_{j \to \infty} u_{j}$.
Assume also that $u \not \equiv \infty$.
Then $u = \w_{p, w} ( f \mu)$.
\item\label{cond:03_potentials}
If $\mu \in \s_{0}$, then $\w_{p, w} \mu =  \w_{p, w}^{0} \mu$.
\end{enumerate}
\end{proposition}

\begin{proof}
\ref{cond:01_potentials}
Let
\[
u'(x) = \lim_{k \to \infty} \w_{p, w}^{0} \mu_{k}(x),
\]
where $\{ \mu_{k} \}_{k = 1}^{\infty}$ is a sequence of Radon measures in Theorem \ref{thm:approximation}.
Since $u' \le u \not \equiv \infty$, it follows from Theorem \ref{thm:TW} that $u'$ is
a nonnegative $(p, w)$-superharmonic solution to $- \laplacian_{p, w} u = \mu$ in $\Omega$.
By Theorem \ref{thm:comparison_principle}, if $\nu \le \mu$ and $\nu \in \s_{00}$, then $\w_{p, w}^{0} \nu \le u'$ in $\Omega$; therefore $u = u'$.
Using the same argument again, we see that $u$ is minimal.
\ref{cond:02_potentials}
Let $\omega$ be the Riesz measure of $u$.
By Theorem \ref{thm:TW}, $f_{k} \mu$ converges to $\omega$ weakly.
By the monotone convergence theorem, $\omega = f \mu$ and $f \mu \in \M_{0}^{+}(\Omega)$.
For each $j \ge 1$, $u_{j} \le \w_{p, w} (f \mu)$, and hence $u \le \w_{p, w} (f \mu)$.
By the minimality of $\w_{p, w}(f \mu)$, $u = \w_{p, w}(f \mu)$.
\ref{cond:03_potentials}
Set $u_{k} = \w_{p, w}^{0} \mu_{k}$.
Since $\mu_{k} \in \s_{00}$ and $\mu \in \s_{0}$, we have
\[
\int_{\Omega} |\nabla u_{k}|^{p} \, dw
=
\int_{\Omega} u_{k} \, d \mu_{k}
\le
\int_{\Omega} u_{k} \, d \mu
\le
C \| \nabla u_{k} \|_{ L^{p}(\Omega; w) }.
\]
Passing to the limit $k \to \infty$, we see that $u = \lim_{k \to \infty} u$ belongs to $H_{0}^{1, p}(\Omega; w)$.
From Theorem \ref{thm:TW} and the uniqueness of weak solutions, the assertion follows.
\end{proof}

\begin{lemma}\label{lem:wmp}
Let $\mu \in \M^{+}_{0}(\Omega)$, and let $u = \w_{p, w} \mu$.
Then, $\sup_{\Omega} u = \| u \|_{L^{\infty}(\mu)}$.
\end{lemma}

\begin{proof}
Take $\{ \mu_{k} \}_{k = 1}^{\infty}$ by using Theorem \ref{thm:approximation}.
Set $u_{k} = \w_{p, w}^{0} \mu_{k}$.
By Remark \ref{rem:wmp}, for any $x \in \Omega$,
\[
u(x)
=
\lim_{k \to \infty} u_{k}(x)
\le
\lim_{k \to \infty} \sup_{\Omega} u_{k}
=
\lim_{k \to \infty} \| u_{k} \|_{L^{\infty}(\mu_{k})}
\le
\| u \|_{L^{\infty}(\mu)}.
\]
The converse inequality is clear.
\end{proof}

\section{Generalized energy}\label{sec:generalized_energy}


Let us define generalized $p$-energy by using minimal $p$-superharmonic solutions.
For $p = 2$ or $\Omega = \R^{n}$, we refer to \cite{MR3881877,HARA2020111847}.

\begin{definition}\label{def:generalized_energy}
For $0 \le \gamma < \infty$, set
$\s^{\gamma} := \left\{ \mu \in \M^{+}_{0}(\Omega) \colon \e_{\gamma}( \mu ) < \infty \right\}$,
where
\[
\e_{\gamma}(\mu)
:=
\int_{\Omega} (\w_{p, w} \mu)^{\gamma} \, d \mu.
\]
Furthermore, let 
$
\s^{\infty}
:=
\left\{
\mu \in \M^{+}_{0}(\Omega) \colon \| \w_{p, w} \mu \|_{L^{\infty}(\mu)} < \infty
\right\}
$
for $\gamma = \infty$.
\end{definition}

By definition, if $\mu \in \s^{\gamma}$, then the Dirichlet problem \eqref{eqn:poisson} has
a minimal nonnegative $(p, w)$-superharmonic solution $u = \w_{p, w} \mu$.
We can verify $\s^{1} = \s_{0}$ by modifying the proof of Theorem \ref{thm:embedding} below.
For any $0 \le \gamma \le \infty$,
\[
\s_{00} = \s^{0} \cap \s^{\infty} \subset \s^{\gamma} \subset \M^{+}_{0}(\Omega).
\]
Assume that $0 < \gamma < \infty$ and $\mu \in \s_{00}$.
Let $u = \w_{p, w}^{0} \mu$ and let $v = u^{\frac{p - 1 + \gamma}{p}}$.
Then, for any $\epsilon > 0$,
\[
\int_{\Omega} (u^{\gamma} - \epsilon)_{+} \, d \mu
=
\gamma \int_{ \{ u^{\gamma} > \epsilon \} } |\nabla u|^{p} u^{\gamma - 1} \, dw
=
\frac{1}{c_{E}} \int_{ \Omega } |\nabla v_{\epsilon}|^{p} \, dw,
\]
where $v_{\epsilon} = (v - \epsilon^{ \frac{p - 1 + \gamma}{\gamma p} })_{+}$ and
$c_{E}$ is the constant in \eqref{eqn:opt-const}.
Thus, by the monotone convergence theorem, 
\begin{equation}\label{eqn:integrating_by_parts}
\int_{\Omega} (\w_{p, w}^{0} \mu)^{\gamma} \, d \mu
=
\gamma \int_{\Omega} |\nabla u|^{p} u^{\gamma - 1} \, dw
=
\frac{1}{c_{E}} \int_{ \Omega } |\nabla v|^{p} \, dw.
\end{equation}
Moreover, since $\{ v_{\epsilon} \}$ is a bounded sequence in $H_{0}^{1, p}(\Omega; w)$,
$v \in H_{0}^{1, p}(\Omega; w)$.

\begin{proposition}\label{prop:energy_class}
Let $\mu \in \M^{+}_{0}(\Omega)$, and let $u = \w_{p, w} \mu$.
\begin{enumerate}[label=(\roman*)]
\item\label{statement:02@energy_class}
If $\mu \in \s^{\gamma}$ with $0 \le \gamma \le 1$,
then $\min \{ u, l \} \in H_{0}^{1, p}(\Omega; w)$ for all $l > 0$.
\item\label{statement:03@energy_class}
If $\mu \in \s^{\gamma}$ with $1 \le \gamma \le \infty$,
then $u$ belongs to $H^{1, p}_{\loc}(\Omega; w)$ and
satisfies $- \laplacian_{p, w} u = \mu$ in the sense of weak solutions.
\item\label{statement:04@energy_class}
Assume that $\mu \in \s^{\gamma_{0}} \cap \s^{\gamma_{1}}$ with
$0 \le \gamma_{0} \le 1$ and $1 \le \gamma_{1} \le \infty$.
Then $u = \w_{p, w}^{0} \mu \in  H_{0}^{1, p}(\Omega; w)$.
In particular, $u$ satisfies \eqref{eqn:poisson} in the sense of finite energy weak solutions.
\end{enumerate}
\end{proposition}

\begin{proof}
Let $\{ \mu_{k} \}_{k = 1}^{\infty}$ be a sequence of Radon measures in Theorem \ref{thm:approximation},
and let $u_{k} = \w_{p, w}^{0} \mu_{k}$.
\ref{statement:02@energy_class}
Testing the equation of $u_{k}$ with $\min \{ u_{k}, l \} \in H_{0}^{1, p}(\Omega; w)$, we get
\begin{equation}\label{eqn:grad_esti}
\int_{\Omega} |\nabla \min \{ u_{k}, l \}|^{p} dw
=
\int_{\Omega} \min \{ u_{k}, l \} \, d \mu_{k}
\le
l^{1 - \gamma}
\int_{\Omega} (\w_{p, w} \mu)^{\gamma} \, d \mu.
\end{equation}
Letting $k \to \infty$ gives the desired boundary condition.
\ref{statement:03@energy_class}
Assume that $1 \le \gamma < \infty$.
Take a ball $B$ such that $B \Subset \Omega$.
Without loss of generality, we may assume that $u_{k_{0}} \neq 0$ for some $k_{0} \ge 1$.
By the strong minimum principle, $\inf_{ \cl{B} } u_{k_{0}} > 0$.
Let $k \geq k_{0}$.
Subsequently, \eqref{eqn:integrating_by_parts} gives
\[
\begin{split}
\int_{B} |\nabla u_{k}|^{p} \, dw
& \le
\frac{1}{\gamma} \left( \inf_{ \cl{B} } u_{k} \right)^{1 - \gamma}
\int_{\Omega} (\w_{p, w} \mu_{k})^{\gamma} \, d \mu_{k}
\\
& \le
\frac{1}{\gamma} \left( \inf_{ \cl{B} } u_{k_{0}} \right)^{1 - \gamma}
\int_{\Omega} (\w_{p, w} \mu)^{\gamma} \, d \mu.
\end{split}
\]
On the other hand, by the Poincar\'{e} inequality,
\[
\begin{split}
\left( \int_{B} u_{k}^{p} \, d w \right)^{\frac{p - 1 + \gamma}{p}}
& \le
w(B)^{ \frac{\gamma - 1}{p} }
\int_{\Omega} u_{k}^{p - 1 + \gamma} \, d w
\le
C \int_{\Omega} |\nabla u_{k}^{ \frac{p - 1 + \gamma}{p} }|^{p} \, d w
\\
& =
c_{E} C \int_{\Omega} (\w_{p, w} \mu_{k})^{\gamma} \, d \mu_{k}
\le
c_{E} C \int_{\Omega} (\w_{p, w} \mu)^{\gamma} \, d \mu.
\end{split}
\]
Hence, $\| u_{k} \|_{H^{1, p}(B; w)}$ is bounded.
Passing to the limit $k \to \infty$, we see that $u \in H^{1, p}(B; w)$.
Thus, $u \in H^{1, p}_{\loc}(\Omega; w)$.
If $\mu \in \s^{\infty}$, then $u$ is a bounded $(p, w)$-superharmonic function.
Therefore, $u \in H^{1, p}_{\loc}(\Omega; w)$.
\ref{statement:04@energy_class}
By H\"{o}lder's inequality, $\s^{\gamma_{0}} \cap \s^{\gamma_{1}} \subset \s^{1} = \s_{0}$.
Hence, $u = \w_{p, w} \mu \in H_{0}^{1, p}(\Omega; w)$.
\end{proof}

As in the proof of \cite[Lemma 3.3]{HARA2020111847},
the Picone-type inequality in \cite{MR1618334,MR3273896} yields the following estimate.

\begin{lemma}\label{lem:trace_estimate}
Let $1 < p < \infty$, $0 < \gamma < \infty$ and $- \gamma < q < p - 1$.
Suppose that $\nu \in \s_{00}$ and $v = \w_{p, w} \nu$.
Assume that $u \in H_{0}^{1, p}(\Omega; w) \cap L^{\infty}(\Omega)$ and $u \ge 0$ in $\Omega$.
Then,
\[
\begin{split}
\int_{\Omega} \tilde{u}^{\gamma + q} \, d \nu
& \le
\left(
\left( \frac{p - 1 + \gamma}{p} \right)^{p}
\int_{\Omega} |\nabla u|^{p} u^{\gamma - 1} \, d w
\right)^{ \frac{\gamma + q}{p - 1 + \gamma} }
\\
& \quad
\times 
\left(
\int_{\Omega} v^{\frac{(\gamma + q)(p - 1)}{p - 1 - q}} \, d \nu
\right)^{ \frac{p - 1 - q}{p - 1 + \gamma} }.
\end{split}
\]
\end{lemma}


\begin{proof}[Proof of Theorem \ref{thm:embedding}]
We first prove the upper bound of $C_{T}$.
Take a sequence of measures $\{ \sigma_{k} \}_{k = 1}^{\infty} \subset \s_{00}$ by using Theorem \ref{thm:approximation}.
Apply Lemma \ref{lem:trace_estimate} to $u_{k} = \min\{ |u|, k \}$ and $\sigma_{k}$; we get
\[
\begin{split}
\int_{\Omega} u_{k}^{1 + q} \, d \sigma_{k}
\le
\left( \int_{\Omega} |\nabla u_{k}|^{p}  \, d w \right)^{ \frac{1 + q}{p} }
\left(
\int_{\Omega} (\w_{p, w} \sigma_{k})^{\frac{(1 + q)(p - 1)}{p - 1 - q}} \, d \sigma_{k}
\right)^{ \frac{p - 1 - q}{p} }.
\end{split}
\]
The desired estimate then follows from the monotone convergence theorem.
Let us prove the lower bound.
Take $\{ \sigma_{k} \}_{k = 1}^{\infty} \subset \s_{00}$.
Let $u = \w_{p, w}^{0} \sigma_{k}$ and let $v = u^{\frac{p - 1}{p - 1 - q}}$.
Note that
\[
\nabla v = \frac{p - 1}{p - 1 - q} \nabla u u^{ \frac{q}{p - 1 - q} } 
\quad
\text{a.e. in $\Omega$.}
\]
Thus, using \eqref{eqn:integrating_by_parts} with $\gamma = \frac{(1 + q)(p - 1) }{p - 1 - q}$, we get
\[
\begin{split}
\int_{\Omega} |\nabla v|^{p} \, d w
& =
\left( \frac{p - 1}{p - 1 - q} \right)^{p}
\int_{\Omega} |\nabla u|^{p} u^{ \frac{pq}{p - 1 - q} } \, dw
\\
& =
\frac{1}{1 + q}
\left( \frac{p - 1}{p - 1 - q} \right)^{p - 1}
\int_{\Omega} u^{ \frac{(1 + q)(p - 1)}{p - 1 - q} } \, d \sigma_{k}.
\end{split}
\]
By density, \eqref{eqn:trace_ineq} gives
\[
\begin{split}
&
\left(
\int_{\Omega} u^{ \frac{(1 + q)(p - 1)}{p - 1 - q} } \, d \sigma_{k}
\right)^{\frac{1}{1 + q}}
\le
\| v \|_{L^{1 + q}(\sigma)}
\le
C_{T} \| \nabla v \|_{L^{p}(w)}
\\
\quad & =
\left( \frac{1}{1 + q} \right)^{ \frac{1}{p} }
\left( \frac{p - 1}{p - 1 - q} \right)^{ \frac{p - 1}{p} }
C_{T}
\left(
\int_{\Omega} u^{ \frac{(1 + q)(p - 1)}{p - 1 - q} } \, d \sigma_{k}
\right)^{\frac{1}{p}}.
\end{split}
\]
Therefore,
\[
\int_{\Omega} \left( \w_{p, w} \sigma_{k} \right)^{ \frac{(1 + q)(p - 1)}{p - 1 - q} } \, d \sigma_{k}
\le
\left( \frac{1}{1 + q} \right)^{ \frac{1 + q}{p - 1 - q} } \frac{1}{c_{V}^{1 + q}} \,
C_{T}^{ \frac{(1 + q)p}{p - 1 - q} }.
\]
Passing to the limit $k \to \infty$, we arrive at the desired lower bound.
\end{proof}

\begin{corollary}\label{cor:energy_esti}
Let $\mu \in \M^{+}_{0}(\Omega)$ and let $0 < \gamma < \infty$.
Then, $\mu \in \s^{\gamma}$ if and only if $v := ( \w_{p, w} \mu )^{ \frac{p - 1 + \gamma}{p} } \in H_{0}^{1, p}(\Omega; w)$.
Moreover,
\[
\int_{\Omega} (\w_{p, w} \mu)^{\gamma} \, d \mu
\le
\int_{ \Omega } |\nabla v|^{p} \, dw
\le
c_{E} \int_{\Omega} (\w_{p, w} \mu)^{\gamma} \, d \mu.
\]
\end{corollary}

\begin{proof}
Let $\{ \mu_{k} \}_{k = 1}^{\infty}$ be a sequence of Radon measures in Theorem \ref{thm:approximation}.
Set $u_{k} = \w_{p, w}^{0} \mu_{k}$ and $v_{k} = \left( u_{k} \right)^{ \frac{p - 1 + \gamma}{p} }$.
Assume that $\mu \in \s^{\gamma}$.
Then, by \eqref{eqn:integrating_by_parts}, $\{ v_{k} \}_{k = 1}^{\infty}$ is a bounded sequence in $H_{0}^{1, p}(\Omega; w)$.
Since $v_{k} \uparrow v$, $v$ belongs to $H_{0}^{1, p}(\Omega; w)$ and satisfies the latter inequality.
Conversely, assume that $v \in H_{0}^{1, p}(\Omega; w)$.
Then, by Theorem \ref{thm:embedding},
\[
\int_{\Omega} u_{k}^{\gamma} \, d \mu_{k}
\le
\int_{\Omega} v^{1 + q} \, d \mu_{k}
\le
\left( \int_{\Omega} |\nabla v|^{p} \, dw \right)^{\frac{1 + q}{p}}
\left( \int_{\Omega} u_{k}^{\gamma} \, d \mu_{k} \right)^{ \frac{p - 1 - q}{p} },
\]
where $q = \frac{\gamma p}{p - 1 + \gamma} - 1$.
Thus using the monotone convergence theorem, we obtain
\[
\int_{\Omega} (\w_{p, w} \mu)^{\gamma} \, d \mu
=
\lim_{k \to \infty} \int_{\Omega} u_{k}^{\gamma} \, d \mu_{k}
\le
\int_{\Omega} |\nabla v|^{p} \, dw.
\]
This completes the proof.
\end{proof}

We also obtain the following estimate using a similar approximation argument.

\begin{theorem}\label{thm:MEE}
Let $1 < p < \infty$, $0 < \gamma < \infty$ and $- \gamma < q < p - 1$.
Then, for any $\mu, \nu \in \M^{+}_{0}(\Omega)$,
\begin{equation}\label{eqn:01@MEE}
\int_{\Omega} ( \w_{p, w} \mu )^{\gamma + q} \, d \nu
\le
\left(
c_{E} \int_{\Omega} (\w_{p, w} \mu)^{\gamma} \, d \mu
\right)^{ \frac{\gamma + q}{p - 1 + \gamma} }
\e_{ \frac{(\gamma + q)(p - 1)}{p - 1 - q} }(\nu)^{\frac{p - 1 - q}{p - 1 + \gamma}}.
\end{equation}
In particular,
if $\mu \in \s^{\gamma}$ and $\nu \in \s^{\frac{ (\gamma + q)(p - 1) }{ p - 1 - q}}$,
then $\w_{p, w} \mu \in L^{\gamma + q}(\nu)$.
\end{theorem}

\begin{remark}
Note that $c_{E} = 1$ if $\gamma = 1$.
This sharp constant is achieved when $\mu = \nu$ is an equilibrium measure.
\end{remark}


\begin{remark}\label{rem:quasi-additivity}
For $0 < \gamma < \infty$, set $\trinorm{\mu}_{\gamma} = \e_{\gamma}(\mu)^{\frac{p - 1}{p - 1 + \gamma}}$.
Then by Theorem \ref{thm:MEE},
\[
\begin{split}
&
\int_{\Omega} \w_{p, w}(\mu + \nu)^{\gamma} d (\mu + \nu)
=
\int_{\Omega} \w_{p, w}(\mu + \nu)^{\gamma} d \mu
+
\int_{\Omega} \w_{p, w}(\mu + \nu)^{\gamma} d \nu
\\
& \quad
\le
c_{E}^{\gamma} \left( \int_{\Omega} \w_{p, w}(\mu + \nu)^{\gamma} d (\mu + \nu) \right)^{\frac{\gamma}{p - 1 + \gamma}}
\left( \trinorm{\mu}_{\gamma} + \trinorm{\nu}_{\gamma} \right),
\end{split}
\]
and hence
\[
\trinorm{\mu + \nu}_{\gamma} \le c_{E}^{\gamma} \left( \trinorm{\mu}_{\gamma} + \trinorm{\nu}_{\gamma} \right).
\]
In particular, $\s^{\gamma}$ is a convex cone.
\end{remark}

Finally, we prove weighted norm inequalities.

\begin{theorem}\label{thm:WNI}
Let $1 < p < \infty$, $0 < q < p - 1$ and $0 < \gamma < \infty$.
Assume that $\sigma \in \s^{ \frac{(\gamma + q)(p - 1)}{p - 1 - q} }$.
Then, for any $f \in L^{ \frac{\gamma + q}{q} }(\sigma)$, $|f| \sigma \in \s^{\gamma}$.
Moreover,
\[
\e_{\gamma}( |f| \sigma )^{\frac{p - 1}{p - 1 + \gamma}}
\le
\left(
c_{E} \, \e_{ \frac{(\gamma + q)(p - 1)}{p - 1 - q} }(\sigma)^{\frac{p - 1 - q}{\gamma + q}}
\right)^{ \frac{\gamma}{p - 1 + \gamma} }
\| f \|_{L^{ \frac{\gamma + q}{q} }(\sigma)}
\]
and
\[
\| \w_{p, w} ( |f| \sigma ) \|_{L^{\gamma + q}( \sigma )}
\le
\left(
c_{E} \, \e_{ \frac{(\gamma + q)(p - 1)}{p - 1 - q} }(\sigma)^{\frac{p - 1 - q}{\gamma + q}}
\right)^{ \frac{1}{p - 1} }
\| f \|_{L^{ \frac{\gamma + q}{q} }(\sigma)}^{\frac{1}{p - 1}}.
\]
\end{theorem}

\begin{proof}
We may assume that $f \ge 0$ without loss of generality.
By Theorem \ref{thm:MEE},
\[
\begin{split}
\int_{\Omega} ( \w_{p, w}(f \sigma) )^{\gamma + q} \, d \sigma
& \le
\left(
c_{E}
\int_{\Omega} (\w_{p, w}(f \sigma))^{\gamma} f  \, d \sigma
\right)^{ \frac{\gamma + q}{p - 1 + \gamma} }
\e_{ \frac{(\gamma + q)(p - 1)}{p - 1 - q} }(\sigma)^{\frac{p - 1 - q}{p - 1 + \gamma}}.
\end{split}
\]
Meanwhile, by H\"{o}lder's inequality,
\[
\int_{\Omega} (\w_{p, w}(f \sigma))^{\gamma} f  \, d \sigma
\le
\left(
\int_{\Omega} (\w_{p, w}(f \sigma))^{\gamma + q} \, d \sigma
\right)^{\frac{\gamma}{\gamma + q}}
\| f \|_{L^{ \frac{\gamma + q}{q} }(\sigma)}.
\]
Combining the two inequalities, we obtain the desired estimates.
\end{proof}

\begin{remark}\label{rem:two-weight}
Under the same assumptions,
suppose also that
$\nu \in \s^{ \frac{(\gamma + Q)(p - 1)}{p - 1 - Q} }$
with
$- \gamma < Q < p - 1$.
Then, the following two weight norm inequality holds:
\[
\| \w_{p, w}( |f| \sigma) \|_{L^{\gamma + Q}(\nu)}
\le
C \| f \|_{L^{ \frac{\gamma + q}{q} }(\sigma)}^{ \frac{1}{p - 1} },
\quad
\forall f \in L^{\frac{\gamma + q}{q}}(\sigma).
\]
In fact, by Theorem \ref{thm:MEE},
\[
\begin{split}
\int_{\Omega} ( \w_{p, w}(f \sigma) )^{\gamma + Q} \, d \nu
& \le
C
\left(
\int_{\Omega} (\w_{p, w}(f \sigma))^{\gamma} f  \, d \sigma
\right)^{ \frac{\gamma + Q}{p - 1 + \gamma} }
\e_{ \frac{(\gamma + Q)(p - 1)}{p - 1 - Q} }(\nu)^{\frac{p - 1 - Q}{p - 1 + \gamma}}.
\end{split}
\]
The right-hand side is estimated by Theorem \ref{thm:WNI}. 
\end{remark}

\section{Properties of solutions to \eqref{eqn:p-laplace}}\label{sec:proof_of_main_theorem}

First, we give the counterpart of \cite[Lemma 3.5]{MR3567503} or \cite[Remark 2.6]{MR4105916}.

\begin{lemma}\label{lem:iterated_ineq}
Let $\sigma \in \M_{0}(\Omega)$, and let $\beta \ge 1$.
Assume that $\left( \w_{p, w} \sigma \right)^{(\beta - 1)(p - 1)} \sigma \in \M_{0}(\Omega)$.
Then,
\[
\left( \w_{p, w} \sigma \right)^{\beta}(x)
\le
\beta \, \w_{p, w} \left( \left( \w_{p, w} \sigma \right)^{(\beta - 1)(p - 1)} \sigma \right)(x),
\quad
\forall x \in \Omega.
\]
\end{lemma}

\begin{proof}
By Proposition \ref{prop:potentials},
we may assume that $\sigma \in \s_{00}$ without loss of generality.
We use the argument in \cite[Lemma 4.4]{HARA2020111847}.
Let $u = \w_{p, w} \sigma$.
Since $u$ is bounded on $\Omega$,
$u^{\beta} \in H_{0}^{1, p}(\Omega; w)$
and
$u^{(\beta - 1)(p - 1)} \sigma \in \s_{00}$.
Fix a nonnegative function $\varphi \in C_{c}^{\infty}(\Omega)$.
Testing the equation of $u$ with $\varphi (u^{(\beta - 1)(p - 1)} - \epsilon)_{+}$, we find that
\[
\begin{split}
\int_{\Omega} \varphi (u^{(\beta - 1)(p - 1)} - \epsilon)_{+} \, d \sigma
& =
\int_{\Omega}
|\nabla u|^{p - 2} \nabla u \cdot \nabla \left( \varphi (u^{(\beta - 1)(p - 1)} - \epsilon)_{+} \right)
\, dw
\\
& \ge
\beta^{1 - p}
\int_{ \{ u^{(\beta - 1)(p - 1)} > \epsilon \} }
|\nabla u^{\beta}|^{p - 2} \nabla u^{\beta}
\cdot \nabla \varphi
\, dw.
\end{split}
\]
Applying the dominated convergence theorem to the right-hand side, we obtain
\[
\begin{split}
\int_{\Omega} \varphi u^{(\beta - 1)(p - 1)} \, d \sigma
\ge
\beta^{1 - p}
\int_{\Omega}
|\nabla u^{\beta}|^{p - 2} \nabla u^{\beta}
\cdot \nabla \varphi
\, dw.
\end{split}
\]
By the comparison principle for weak solutions, this implies that
\[
u^{\beta}
\le
\w_{p, w}^{0}( \beta^{p - 1} u^{(\beta - 1)(p - 1)} \sigma)
\quad
\text{a.e. in $\Omega$.}
\]
Since $u$ is $(p, w)$-superharmonic in $\Omega$, the desired inequality holds.
\end{proof}

Next, we give the counterpart of \cite[Theorem 3.4]{MR3567503} or \cite[Theorem 1.3]{MR4105916}.

\begin{theorem}\label{thm:lower_bound}
Let $1 < p < \infty$ and $0 < q < p - 1$.
Let $\sigma \in \M^{+}_{0}(\Omega)$.
Let  $v \in L^{q}_{\loc}(\sigma)$ be a nontrivial nonnegative $(p, w)$-superharmonic supersolution to
$- \laplacian_{p, w} v = \sigma v^{q}$ in $\Omega$.
Then,
\[
v(x) \ge c_{V} \left( \w_{p, w} \sigma \right)^{ \frac{p - 1}{p - 1 - q} }(x),
\quad
\forall x \in \Omega,
\]
where $c_{V}$ is the constant in \eqref{eqn:opt-const}.
\end{theorem}


\begin{proof}
For simplicity, we write $\w_{p, w} \mu$ as $\w \mu$.
Let $u = \w(v^{q} \sigma)$.
By Theorem \ref{thm:comparison_principle}, $v(x) \ge u(x)$ for all $x \in \Omega$.
Fix $a > 0$, and set $\sigma_{a} = \mathbf{1}_ { \{ x\in \Omega \colon u(x) > a \} } \sigma$.
Using Theorem \ref{thm:comparison_principle} again, we get
\[
u
\ge
\w(u^{q} \sigma)
\ge
\w(a^{q} \sigma_{a} )
=
a^{\frac{q}{p - 1}} \w \sigma_{a}.
\]
Continuing this argument $k$-times, we obtain
\begin{equation}\label{eqn:01@thm:lower_bound}
\begin{split}
u
& \ge
\w( \w( \cdots ( \w( u^{q} \sigma) )^{q} \cdots \sigma)^{q} \sigma)
\\
& \ge
\w( \w( \cdots ( \w( a^{q} \sigma_{a}) )^{q} \cdots \sigma )^{q} \sigma )
\\
& \ge
a^{ (\frac{q}{p - 1})^{k} }
\w( \w( \cdots ( \w \sigma_{a} )^{q} \cdots \sigma_{a} )^{q} \sigma_{a} ).
\end{split}
\end{equation}
Meanwhile, by Lemma \ref{lem:iterated_ineq}, 
\[
\left( \w \sigma_{a} \right)^{\beta_{i + 1}}
\le
\beta_{i + 1} \, \w \left( \left( \w \sigma_{a} \right)^{\beta_{i} q} \sigma_{a} \right)
\]
for each $i \ge 0$, where $\beta_{0} = 1$ and $\beta_{i + 1} = \beta_{i} \frac{q}{p - 1} + 1$.
Iterating this estimate $k$ times, we get
\begin{equation}\label{eqn:02@thm:lower_bound}
(\w \sigma_{a})^{\beta_{k}} 
\le
\prod_{i = 1}^{k} \beta_{i}^{ \left( \frac{q}{p - 1} \right)^{k - i} }
\w( \w( \cdots ( \w \sigma_{a} )^{q} \cdots \sigma_{a} )^{q} \sigma_{a} ).
\end{equation}
By definition, $\beta_{k} = \sum_{i = 0}^{k} \left( \frac{q}{p - 1} \right)^{i}$.
Therefore $\beta_{k} \uparrow \frac{p - 1}{p - 1 - q} $ as $k \to \infty$ and
\begin{equation}\label{eqn:03@thm:lower_bound}
\prod_{i = 1}^{k} \beta_{i}^{ \left( \frac{q}{p - 1} \right)^{k - i} }
\le
\left( \frac{p - 1}{p - 1 - q}  \right)^{ \sum_{i = 1}^{k} \left( \frac{q}{p - 1} \right)^{k - i}  }
\le
\left( \frac{p - 1}{p - 1 - q}  \right)^{ \frac{p - 1}{p - 1 - q} }.
\end{equation}
Combining \eqref{eqn:01@thm:lower_bound}, \eqref{eqn:02@thm:lower_bound} and \eqref{eqn:03@thm:lower_bound}
and letting $k \to \infty$, we obtain
\[
u
\ge
\left( \frac{p - 1 - q}{p - 1}  \right)^{ \frac{p - 1}{p - 1 - q} }
(\w \sigma_{a})^{ \frac{p - 1}{p - 1 - q} }.
\]
Without loss of generality, we may assume that $\sigma_{a} \neq 0$ for small $a > 0$.
Then $u \ge a^{\frac{q}{p - 1} }\w \sigma_{a} > 0$ in $\Omega$ by the strong minimum principle.
Thus, taking the limit $a \to 0$, we arrive at the desired estimate.
\end{proof}

Finally, we prove Theorem \ref{thm:main_theorem} and its variants.

\begin{proof}[Proof of Theorem \ref{thm:main_theorem}]

\ref{enum:01@main_theorem} $\Rightarrow$ \ref{enum:02@main_theorem}:
Using Theorem \ref{thm:lower_bound}, we find that
\[
\begin{split}
\left(
\int_{\Omega} \left( \w_{p, w} \sigma \right)^{ \frac{(\gamma + q)(p - 1)}{p - 1 - q} } \, d \sigma
\right)^{\frac{1}{\gamma + q}}
& \le
\frac{1}{c_{V}}
\| v \|_{L^{\gamma + q}(\sigma)}
\le
\frac{C_{1}}{c_{V}}.
\end{split}
\]

\ref{enum:02@main_theorem} $\Rightarrow$ \ref{enum:03@main_theorem}:
By Theorem \ref{thm:WNI}, \eqref{eqn:weighted_norm_ineq} holds with
\[
\begin{split}
C_{3}^{ \frac{p - 1}{p - 1 - q}}
\le
\left(
c_{E} \e_{ \frac{(\gamma + q)(p - 1)}{p - 1 - q} }(\sigma)^{\frac{p - 1 - q}{\gamma + q}}
\right)^{ \frac{1}{p - 1 - q} }
\le
c_{E}^{ \frac{1}{p - 1 - q} } C_{2}.
\end{split}
\]

\ref{enum:03@main_theorem} $\Rightarrow$ \ref{enum:01@main_theorem}:
Take $\{ \mathbf{1}_{F_{k}} \sigma \}_{k = 1}^{\infty} \subset \s_{00}$ by using Theorem \ref{thm:approximation}.
Applying \eqref{eqn:weighted_norm_ineq} to
$f = \left( \w \sigma_{k} \right)^{ \frac{q(p - 1)}{p - 1 - q} } \mathbf{1}_{F_{k}} \in L^{ \frac{\gamma + q}{q}}(\sigma)$
and using Lemma \ref{lem:iterated_ineq},
we obtain
\[
\begin{split}
\left(
\int_{\Omega} \left( \w_{p, w} \sigma_{k} \right)^{ \frac{(\gamma +q)(p - 1)}{p - 1 - q} } \mathbf{1}_{F_{k}} \, d \sigma
\right)^{\frac{1}{\gamma + q}}
\le
C.
\end{split}
\]
Therefore, by the monotone convergence theorem,
\[
u_{0} := c_{V} \left(\w_{p, w} \sigma \right)^{\frac{p - 1}{p - 1 - q}} \in L^{ \gamma + q }(\sigma),
\]
where $c_{V}$ is the constant in \eqref{eqn:opt-const}.
Define a sequence of $(p, w)$-superharmonic functions $\{ u_{i} \}_{i = 1}^{\infty}$ by 
\[
u_{i + 1} := \w_{p, w} (u_{i}^{q} \sigma), \quad i = 1, 2, \dots.
\]
By \eqref{eqn:weighted_norm_ineq},
\[
\begin{split}
\| u_{i + 1} \|_{L^{\gamma + q}(\sigma)}
=
\| \w_{p, w} ( u_{i}^{q} \sigma) \|_{L^{\gamma + q}(\sigma)}
\le
C_{3} \| u_{i} \|_{L^{ \gamma + q }(\sigma)}^{ \frac{q}{p - 1} },
\end{split}
\]
and hence $\{ u_{i} \}_{i = 1}^{\infty} \subset L^{\gamma + q}(\sigma) \subset L^{q}_{\loc}(\sigma)$.
By Lemma \ref{lem:iterated_ineq}, $u_{0} \le u_{1}$.
Hence,  by induction, $u_{i} \le u_{i + 1}$ for all $i \ge 1$.
Let $u = \lim_{i \to \infty} u_{i}$.
By the monotone convergence theorem,
\[
\| u \|_{L^{\gamma + q}(\sigma)} 
=
\lim_{i \to \infty} \| u_{i + 1} \|_{L^{\gamma + q}(\sigma)}
\le
C_{3}^{\frac{p - 1}{p - 1 - q}}.
\]
By Proposition \ref{prop:potentials}, $u$ is $(p, w)$-superharmonic in $\Omega$ and $u = \w_{p, w} ( u^{q} \sigma)$.

Assume that $v$ is a nontrivial nonnegative $(p,w)$-superharmonic solution to
$- \laplacian_{p, w} v = \sigma v^{q}$ in $\Omega$.
Then $u_{0} \le v$ by Theorem \ref{thm:lower_bound}, and hence, $u_{i} \le v$ for all $i \ge 1$ by induction.
Therefore $u \le v$.
\end{proof}

\begin{remark}
In Theorem \ref{thm:main_theorem},
the equivalence \ref{enum:01@main_theorem} $\Leftrightarrow$ \ref{enum:02@main_theorem}
still holds even if $q = 0$.
\end{remark}

\begin{theorem}\label{thm:main_theorem_infty}
Let $\Omega$ be a bounded domain in $\R^{n}$.
Let $1 < p < \infty$ and $0 < q < p - 1$.
Suppose that $\sigma \in \M^{+}_{0}(\Omega) \setminus \{ 0 \}$.
Then, the following statements are equivalent:
\begin{enumerate}[label=(\arabic*)]
\item\label{enum:01@main_theorem_infty}
There exists a bounded positive weak supersolution
$v \in H^{1, p}_{\loc}(\Omega; w)$ to $- \laplacian_{p, w} v = \sigma v^{q}$ in $\Omega$
satisfying $\| v \|_{ L^{\infty}(\sigma) } \le C_{1} < \infty$.
\item\label{enum:02@main_theorem_infty}
$\| \w_{p, w} \sigma \|_{L^{\infty}(\sigma)}^{ \frac{p - 1}{p - 1 - q} } \le C_{2} < \infty$.
\item\label{enum:03@main_theorem_infty}
The following weighted norm inequality holds:
\[
\| \w_{p, w} ( |f| \sigma ) \|_{L^{\infty}( \sigma )}
\le
C_{3}  \| f \|_{L^{ \infty }( \sigma )}^{\frac{1}{p - 1}},
\quad
\forall f \in L^{\infty}( \sigma ).
\]
\end{enumerate}
Moreover, if $C_{i}$ ($i = 1, 2 ,3$) are the best constants in the above statements, then
\[
C_{1} \le C_{3}^{ \frac{p - 1}{p - 1 - q} } \le C_{2} \le \frac{ C_{1} }{c_{V}}.
\]
In addition, if one of the above statements holds,
then there exists a minimal positive $(p, w)$-superharmonic solution $u$ to
$- \laplacian_{p, w} u = \sigma u^{q}$ in $\Omega$
such that $\sup_{\Omega} u \le C_{1}$.
\end{theorem}

\begin{proof}
We modify the Proof of Theorem \ref{thm:main_theorem}.
Clearly, \ref{enum:01@main_theorem_infty} $\Rightarrow$ \ref{enum:02@main_theorem_infty} $\Rightarrow$ \ref{enum:03@main_theorem_infty}.
Assume that \ref{enum:03@main_theorem_infty} holds. Then, for all $i \ge 1$,
\[
\begin{split}
\| u_{i + 1} \|_{L^{\infty}(\sigma)}
=
\| \w_{p, w} ( u_{i}^{q} \sigma) \|_{L^{\infty}(\sigma)}
\le
C_{3} \| u_{i} \|_{L^{\infty}(\sigma)}^{ \frac{q}{p - 1} }.
\end{split}
\]
Thus,
\[
\| u \|_{L^{\infty}(u^{q} \sigma)} 
\le
\| u \|_{L^{\infty}(\sigma)}
\le
C_{3}^{ \frac{p - 1}{p - 1 - q} }.
\]
By Lemma \ref{lem:wmp},
$\sup_{\Omega} u \le C_{3}^{ \frac{p - 1}{p - 1 - q} }$ and \ref{enum:01@main_theorem_infty} holds.
\end{proof}

\begin{proof}[Proof of Theorem \ref{thm:main_theorem_FE}]
The case of $q = 0$ is \cite[Corollary 21.18]{MR2305115}, so we consider $q > 0$.
Assume that a positive finite energy solution $u$ exists.
We may assume that $u$ is quasicontinuous without loss of generality.
Then, we find that
\[
\begin{split}
\| \nabla u \|_{L^{p}(w)}^{p}
& =
\int_{\Omega} |\nabla u|^{p} \, d w
=
\int_{\Omega} u^{1 + q} \, d \sigma
=
\| u \|_{L^{1 + q}(\sigma)}^{1 + q}.
\end{split}
\]
Therefore, \eqref{eqn:energy_cond@FE} follows from Theorem \ref{thm:main_theorem}.
Conversely, assume that \eqref{eqn:energy_cond@FE} holds.
Then, Theorem \ref{thm:main_theorem} and Proposition \ref{prop:energy_class} give
a minimal positive finite energy weak solution $u \in H_{0}^{1, p}(\Omega; w)$ to $- \laplacian_{p, w} u = \sigma u^{q}$ in $\Omega$.
The uniqueness of such a solution follows from a convexity argument as in \cite[Theorem 5.1]{MR3311903}.
\end{proof}

\begin{remark}
As the proof of \cite[Corollary 6.3]{HARA2020111847},
using Remark \ref{rem:quasi-additivity},
we can construct weak solutions to quasilinear equations
of the type
\[
\begin{cases}
\displaystyle
- \laplacian_{p, w} u = \sum_{j = 1}^{J} \sigma_{j} u^{q_{j}} 
& \text{in} \ \Omega,
\\
u = 0
& \text{on} \ \del \Omega,
\end{cases}
\]
where $0 \le q_{j} < p - 1$ and $\sigma_{j} \in \s^{ \frac{ (1 + q_{j})(p - 1) }{p - 1 - q_{j}} }$ for $j = 1, 2, \dots, J$.
\end{remark}

\section{Quasilinear PDE with $L^{s, t}$ coefficients}\label{sec:example1}

Let us assume that $\mu \in \M^{+}_{0}(\Omega)$ is finite.
Then as in \cite[Theorem 2.1]{MR2456885},
\begin{equation}\label{eqn:wolff_UB}
\w_{p, w} \mu(x)
\le
C \W_{1, p, w}^{ 2 \diam(\Omega) } \mu(x),
\quad 
\forall x \in \Omega.
\end{equation}
Using this, we can estimate the generalized $p$-energy of $\mu$.

We now consider unweighted equations.
As the usual notation, we write $H_{0}^{1, p}(\Omega; 1)$ as $W_{0}^{1, p}(\Omega)$.
For a Lebesgue measurable function $f$ on $\Omega$, we define
\[
\| f \|_{L^{r, \rho}(\Omega)}
=
\begin{cases}
\displaystyle
\left( \int_{0}^{\infty} \left( t^{\frac{1}{r}} f^{*}(t) \right)^{\rho} \frac{dt}{t} \right)^{\frac{1}{\rho}}
& \text{if} \ \rho < \infty, \\
\displaystyle
\sup_{t > 0} t^{\frac{1}{r}} f^{*}(t)
& \text{if} \ \rho = \infty,
\end{cases}
\]
where $0 < r, \rho \leq \infty$ and
$f^{*}(t) = \inf \{ \alpha > 0 \colon | \{ x \in \Omega \colon |f(x)| > \alpha \} | \leq t \}$.
The space of all $f$ with $\| f \|_{L^{r, \rho}(\Omega)} < \infty$ is called the \textit{Lorentz space}.
For the basics of Lorentz spaces, we refer to \cite[Chapter 1]{MR2445437}.

\begin{corollary}\label{cor:BO}
Let $1 < p < n$, $0 \le q < p - 1$ and $0 < \gamma \le \infty$.
Set
\[
s = \frac{n(p - 1 + \gamma)}{n (p - 1 - q) + p(\gamma + q)}
\quad \text{and} \quad
t = \frac{p - 1 + \gamma}{p - 1 - q}.
\]
Let $\sigma =  \theta \, dx$, where $\theta \neq 0$ is a nonnegative function in $L^{s, t}(\Omega)$.
Then there exists a minimal positive $(p, w)$-superharmonic solution $u$ to \eqref{eqn:p-laplace} such that
$\min\{ u, l \} \in W_{0}^{1, p}(\Omega)$ for all $l > 0$.
Moreover, $u \in L^{r, \rho}(\Omega)$ and
\begin{equation}\label{eqn:00_BO}
\| u \|_{L^{r, \rho}(\Omega)}
\le
C \| \theta \|_{L^{s, t}(\Omega)}^{\frac{1}{p - 1 - q}},
\end{equation}
where 
\[
r = \frac{n (p - 1 + \gamma)}{n - p},
\quad
\rho = p - 1 + \gamma
\]
and $C$ is a positive constant depending only on $n$, $p$, $q$ and $\gamma$.
Assume also that $\gamma \ge 1$.
Then $u$ belongs to $W_{0}^{1, p}(\Omega)$
and satisfies \eqref{eqn:p-laplace} in the sense of weak solutions.
\end{corollary}

\begin{proof}
As in the proof of  \cite[Corollary 3.6]{HARA2020111847},
using \eqref{eqn:wolff_UB} and the Havin-Maz'ya potential estimate (see \cite{MR727526}),
we find that
\[
\int_{\Omega} (\w_{p, 1} (\theta \, dx))^{ \frac{(\gamma + q)(p - 1)}{p - 1 - q} } \theta \, dx
\le
C \| \theta \|_{L^{s, t}(\Omega)}^{\frac{p - 1 + \gamma}{p - 1 - q}}.
\]
By Theorem \ref{thm:main_theorem}, there exists a minimal $p$-superharmonic solution $u$ such that
\begin{equation}\label{eqn:01_BO}
\e_{\gamma}( u^{q} \sigma )
=
\| u \|_{L^{\gamma + q}(\sigma)}^{\gamma + q}
\le
C \| \theta \|_{L^{s, t}(\Omega)}^{\frac{p - 1 + \gamma}{p - 1 - q}}.
\end{equation}
Furthermore, by Corollary \ref{cor:energy_esti}
and a sharp form of Sobolev inequality (see, e.g., \cite[p.234]{MR2777530}),
\begin{equation}\label{eqn:02_BO}
\| u \|_{L^{r, \rho}(\Omega)}^{p - 1 + \gamma}
=
\| v \|_{L^{p^{*}, p}(\Omega)}^{p}
\le
C \int_{\Omega} |\nabla v|^{p} \, dx
\le
C \e_{\gamma}( u^{q} \sigma ),
\end{equation}
where $v = u^{ \frac{p - 1 + \gamma}{p} }$.
Combining \eqref{eqn:01_BO} and \eqref{eqn:02_BO}, we obtain \eqref{eqn:00_BO}.
Since $\Omega$ is bounded, we have
\[
\begin{split}
\int_{\Omega} u^{q} \theta dx
& \le
\left(
\int_{\Omega}
u^{\gamma + q} \theta \,dx
\right)^{\frac{q}{\gamma + q}}
\left(
\int_{\Omega} \theta \, dx
\right)^{\frac{\gamma}{\gamma + q}}
\\
& \le
\left(
\int_{\Omega}
u^{\gamma + q} \theta \,dx
\right)^{\frac{q}{\gamma + q}}
\left(
C \| \theta \|_{L^{s, t}(\Omega)} |\Omega|^{\frac{s - 1}{s}}
\right)^{\frac{\gamma}{\gamma + q}} < \infty.
\end{split}
\]
Thus, $u^{q} \sigma = u^{q} \theta \, dx \in \s^{0}$.
By Proposition \ref{prop:energy_class}, this implies that $\min\{ u, l \} \in W_{0}^{1, p}(\Omega)$ for all $l > 0$.
If $\gamma \ge 1$, then $u^{q} \sigma \in \s^{0} \cap \s^{\gamma} \subset \s^{1}$, and hence $u \in W_{0}^{1, p}(\Omega)$.
\end{proof}

\begin{remark}
For $0 < \gamma < 1$,
using an interpolation argument (see, e.g., \cite[Lemma 4.2]{MR1354907}),
from \eqref{eqn:02_BO} and \eqref{eqn:grad_esti}, we can deduce a gradient estimate of $u$.
\end{remark}

\section{Quasilinear ODE with Hardy-type coefficients}\label{sec:example2}

Let us now consider the model ordinary differential equation
\begin{equation}\label{eqn:p-laplace_bdr}
\begin{cases}
- (w |u'|^{p - 2} u')' = \theta u^{q} \quad \text{in $(-1, 1)$,}
\\
u(-1) = u(1) = 0,
\end{cases}
\end{equation}
where
\begin{align*}
w(x) & = (1 - |x|)^{\beta}, \quad \beta \in (-1, p - 1),
\\
\theta(x) & = (1 - |x|)^{- \alpha}, \quad \alpha \in \R,
\end{align*}
and $' = \frac{d}{dx}$.
The function $w$ can be regarded as a Muckenhoupt $A_{p}$-weight in $\R$.
Therefore it is also $p$-admissible (see \cite[Chapter 15]{MR2305115} \cite[Theorem 2]{MR2180887}).
Note that $\theta$ is not integrable on $(-1, 1)$ if $\alpha \ge 1$.

\begin{corollary}
Let $1 < p < \infty$ and $0 \le q  < p - 1$.
Assume that $\alpha  < p - \beta$.
Then there exists a bounded minimal positive weak solution $u \in H^{1, p}_{\loc}((-1, 1); w) \cap C([-1, 1])$ to \eqref{eqn:p-laplace_bdr}.
Moreover, there exists a positive finite energy weak solution $u \in H_{0}^{1, p}((-1, 1); w)$ to \eqref{eqn:p-laplace_bdr}
if and only if
\begin{equation}\label{eqn:bdr_singularity}
\alpha < 1 + (1 + q) \left( 1 - \frac{1}{p} \right) \left( 1 - \frac{\beta}{p - 1} \right).
\end{equation}
\end{corollary}

\begin{proof}
Let $v = (1 - |x|)^{A}$. Then
\[
- (w |v'|^{p - 2} v')'
=
c(p, \beta, A) (1 - |x|)^{(A - 1)(p - 1) + \beta - 1} + 2A^{p - 1} \delta_{0}
\]
in the sense of distribution, where
$c(p, \beta, A) = - A^{p - 1}\{ (A - 1)(p - 1) + \beta \}$
and 
$\delta_{0}$ is the Dirac mass concentrated at $0$.
Hence taking $A \in (0, 1 - \frac{\beta}{p - 1})$ such that $A \le \frac{p - \alpha - \beta}{p - 1- q}$ and
choosing a large $C$, we can make $V(x) = C (1 - |x|)^{A}$ a bounded supersolution
to \eqref{eqn:p-laplace_bdr}.
Then Theorem \ref{thm:main_theorem_infty} gives a positive bounded weak solution $u \in H^{1, p}_{\loc}((-1, 1); w)$ to \eqref{eqn:p-laplace_bdr}.
The Sobolev embedding theorem provides $u \in C(-1, 1)$.
Furthermore, $u$ is continuous up to the boundary because $0 \le u(x) \le V(x)$ for all $x \in (-1, 1)$.

According to the Hardy-type inequality in \cite[Theorem 1.3.3]{MR2777530},
\eqref{eqn:bdr_singularity} is necessary and sufficient for the embedding
\[
\left( \int_{-1}^{1} |u|^{1 + q} \theta \, dx \right)^{\frac{1}{1 + q}}
\le
C \left( \int_{-1}^{1} |u'|^{p} w \, dx \right)^{\frac{1}{p}},
\quad
\forall u \in C_{c}^{\infty}(-1, 1).
\]
Thus Theorems \ref{thm:main_theorem_FE}  and \ref{thm:embedding} give the desired assertion.
\end{proof}


\bibliographystyle{abbrv}
\bibliography{reference}


\end{document}